\documentclass[letterpaper, 10 pt, conference]{ieeeconf} 
\IEEEoverridecommandlockouts   
\usepackage{cite}
\usepackage{amsmath,amssymb,amsfonts}
\usepackage{algorithmic}
\usepackage{graphicx}
\usepackage{textcomp}
\usepackage{soul}
\usepackage{caption}
\usepackage{subcaption}
\usepackage{color}

\DeclareCaptionLabelFormat{nocaption}{}

\usepackage{amsthm}

\newtheoremstyle{ieeeconf}
{0pt}   % ABOVESPACE
{0pt}   % BELOWSPACE
{\normalfont}  % BODYFONT
{\parindent}       % INDENT (empty value is the same as 0pt)
{\itshape} % HEADFONT
{:}         % HEADPUNCT
{ } % HEADSPACE
{\thmname{#1} \thmnumber{#2}\thmnote{ (#3)}} % CUSTOM-HEAD-SPEC
\makeatletter
\renewenvironment{proof}[1][\proofname]{\par
	\pushQED{\qed}%
	\normalfont \topsep\z@
	\trivlist
	\item[\hskip2em
	\itshape
	#1\@addpunct{:}]\ignorespaces
}{%
	\popQED\endtrivlist\@endpefalse
}
\makeatletter

\theoremstyle{ieeeconf}

\newtheorem{thm}{Theorem}
\newtheorem{defn}{Definition}
\newtheorem{prop}{Proposition}
\newtheorem{lem}{Lemma}

\newtheorem{cor}{Corollary}

\newtheorem{ex}{Example}

\newcommand{\R}{\mathbb{R}}

\newcommand{\N}{\mathbb{N}}

\newcommand{\eps}{\varepsilon}

\title{\LARGE \bf
Existence of Partially Quadratic Lyapunov Functions That Can Certify The Local Asymptotic Stability of Nonlinear Systems}
\author{Morgan Jones,%
	\thanks{M. Jones is with the Department of Automatic Control and Systems Engineering,
		The University of Sheffield, Amy Johnson Building, Mappin Street, Sheffield, S1 3JD. e-mail: {\tt \small morgan.jones@sheffield.ac.uk}   }
	Matthew M. Peet% <-this % stops a space
	\thanks{M. Peet is with the School for the Engineering of Matter, Transport and Energy, Arizona State University, Tempe, AZ, 85298 USA. e-mail: {\tt \small mpeet@asu.edu } }
}

\begin{document}

\maketitle
\thispagestyle{plain}
\pagestyle{plain}
\begin{abstract}
This paper proposes a method for certifying the local asymptotic stability of a given nonlinear Ordinary Differential Equation (ODE) by using Sum-of-Squares (SOS) programming to search for a partially quadratic Lyapunov Function (LF). The proposed method is particularly well suited to the stability analysis of ODEs with high dimensional state spaces. This is due to the fact that partially quadratic LFs are parametrized by fewer decision variables when compared with general SOS LFs. The main contribution of this paper is using the Center Manifold Theorem to show that partially quadratic LFs that certify the local asymptotic stability of a given ODE exist under certain conditions. %A numerical example based on the Generalized Lotka-Volterrra ODE is given to demonstrate the computational savings of the proposed method.
\end{abstract}
\vspace{-0.25cm} \section{Introduction} \vspace{-0.2cm}
There is an abundance and diversity of applications found throughout science where a dynamical system is modelled as a nonlinear Ordinary Differential Equation (ODE). ODEs are at the core of many topics ranging from chaos theory, with the Lorenz equation~\cite{lorenz1963deterministic},  population dynamics~\cite{jansen2000local}, power systems~\cite{https://doi.org/10.48550/arxiv.2111.09382} and many more. Understanding the long term properties of solutions to general ODEs is therefore of critical importance. Arguably the most fundamental and sort after long term property is that of local stability. A system described by an ODE is said to be locally asymptotically stable if solutions initialized near an equilibrium point remain near this equilibrium point for all time and furthermore converge towards this equilibrium point as time increases.

This paper considers the following problem: Given an ODE and its equilibrium point, certify whether or not this ODE is locally asymptotically stable. To solve this problem we take the approach that is perhaps the most universally used technique, Lyapunov's second method. This method certifies the stability of ODEs by finding a function satisfying certain properties called a Lyapunov Function (LF). %Unfortunately, the problem of finding a LF to certify local asymptotic stability is often intractable~\cite{ahmadi2013stability}. 

%In recent years the problem of finding LFs has been viewed as a convex feasibility problem

%Linear ODEs, $\dot{x}(t)=Ax$, are asymptotically stable if and only if there exists a quadratic LF, $V(x)=x^\top Px$, for some $P>0$ that satisfies the following Linear Matrix Inequality (LMI), $A^\top P+ PA <0$. Unfortunately, for the more general case of locally asymptotically stable nonlinear ODEs, there does not necessarily exist a quadratic LF that can certify this stability. 

A common approach to numerically searching for LFs has been to use Sum-of-Square (SOS) programming~\cite{Declan2022}. Unfortunately, searching for SOS LFs is known to scale poorly with respect to the state space dimension of the system~\cite{ahmadi2019dsos}. One approach to improving the scalability of SOS has been to decompose large scale systems into lower dimensional subsystems. Several methods exist that show that if a suitable decomposition can be found then the  stability of the lower dimensional subsystems imply the stability of the original large scale system~\cite{kundu2015sum,anderson2011decomposition,schlosser2020sparse}. Unfortunately, these methods often lack generality assuming that the system has a certain structure that allows for such decompositions.

Recently there has been significant interest in improving the scalability of SOS methods by searching for ``separable" or ``structured" LFs of the form,
\vspace{-0.45cm} \[
V(x)=\sup_{1 \le i \le n} V_i(x_i) \quad  \text{ or } \quad  V(x)=\sum_{i=1}^n V_i(x_i) .
\]
Such separable LFs can be found in the works of~\cite{ito2012capability,ali2020computational,tan2008stability} and in the related reachable set computation problems~\cite{tacchi2020approximating,cibulka2021spatio}. These works demonstrate that searching for ``structured" LFs improve numerical performance. It has been shown in~\cite{rantzer2013separable} that monotone systems over compact state spaces posses max-separable LFs. However, in general, it is unknown for what class systems possess such ``structured" LFs.

Inspired by the works of \cite{Khalil_1996,haasdonk2021kernel} that use the Center Manifold Theorem to construct converse LFs with certain structure, in this paper we propose a new approach for certifying the local asymptotic stability of general high state space nonlinear ODEs by searching for partially quadratic LFs of the form,
\vspace{-0.3cm} \[
V(x_1,x_2)= J(x_1) + x_2^\top H(x_1) + x_2^\top P x_2,
 \]
 where $J:\R^k \to \R$, $H:\R^k \to \R^{(n-k)}$, $P\in \R^{(n-k) \times (n-k)}$ and $k \in \{1,\dots,n\}$. Such LFs are partially quadratic since a subset of the state space variables, $x_2 \in \R^{(n-k)}$, appear in $V$ with degree at most two. The main contribution of this paper is to provide several conditions under which it can be shown that partially quadratic LFs exist.
 
 %The paper is organized as follows. Notation is given in Section~\ref{sec: notation}. The definition of local asymptotic stability and Lyapunovs second method is detailed in Section~\ref{sec: ODE and soln map}. The background material on the center manifold theorem required to prove the existence of partially quadratic LFs is given in Section~\ref{sec: center manifold}. Conditions that imply the existence of partially quadtratic LFs are given in Section~\ref{sec: converse partiall quadratic LF}. Our proposed method for certifying the local stability of large scale systems is given in Section~\ref{sec: SOS problems}. Numerical examples are provided in Section~\ref{sec: numerical examples} and finally our conclusion is given in Section~\ref{sec:conclusion}
\vspace{-0.25cm} \section{Notation} \label{sec: notation} \vspace{-0.1cm}
We denote a ball with radius $R>0$ centred at the origin by $B_R(0)=\{x \in \R^n: x^\top x< R^2\}$. Let $C(X,Y)$ be the space of continuous functions with domain $X \subset \R^n$ and image $Y \subset \R^n$. We denote the set of differentiable functions by $C^i(X,Y):=\{f \in C(X,Y): \Pi_{k=1}^{n} \frac{\partial^{\alpha_k} f}{\partial x_k^{\alpha_k}}  \in C(X,Y) \text{ } \forall \alpha \in \N^n \text{ such that } \sum_{j=1}^{n} \alpha_j \le i\}$. For $V \in C^1(\R^n , \R^m)$ we denote $\nabla V$ as the $n\times m$ matrix function such that $(\nabla V(x))_{i,j}=\frac{\partial V_j}{\partial x_i}(x)$. For $d \in \N$ and $x \in \R^n$ we denote $z_d(x)$ to be the vector of monomial basis in $n$-dimensions with maximum degree $d \in \mathbb{N}$. We denote the space of scalar valued polynomials $p: \R^n \to \R$ with degree at most $d \in \N$ by $\R_d[x]$. We say $p \in \R_d[x]$ is a Sum-of-Squares (SOS) polynomial if there exists $p_i \in \R_d[x]$ such that $p(x) = \sum_{i=1}^{k} (p_i(x))^2$. We denote $\Sigma_{2d}$ to be the set of $2d$-degree SOS polynomials.
\vspace{-0.15cm} \section{ODEs and Solution Maps} \vspace{-0.1cm} \label{sec: ODE and soln map}
Consider a nonlinear Ordinary Differential Equation (ODE) of the form
\vspace{-0.65cm} \begin{equation} \label{eqn: ODE}
	\dot{x}(t) = f(x(t)),
\end{equation}
where $f: \R^n \to \R^n$ is the vector field. Note that, throughout this paper we will assume $f(0)=0$, implying the origin is an equilibrium point of the ODE~\eqref{eqn: ODE}.

\paragraph{The Solution Map of ODEs}
Given $D \subset \R^n$, and $I \subset [0, \infty)$ we say any function $\phi_f :D \times I \to \R^n$ satisfying
\vspace{-0.4cm}\begin{align} \label{soln map property}
	&\frac{\partial \phi_f(x,t)}{\partial t}= f(\phi_f(x,t)), \text{ } \phi_f(x,0)=x \text{ for } (x,t) \in D \times I,
\end{align}
is a solution map of the ODE~\eqref{eqn: ODE} over $D \times I$. For simplicity throughout the paper we will assume there exists a unique solution map to the ODE~\eqref{eqn: ODE} over all $(x,t) \in \R^n \times [0,\infty)$.  {Note, if the vector field, $f$, is Lipschitz continuous then the solution map exists for some finite time interval, furthermore, this finite time interval can be arbitrarily extended if the solution map does not leave some compact set, see~\cite{Khalil_1996}. }
%uniqueness and existence of a solution map sufficient for the purposes of this paper, can be established  such as for initial conditions inside some invariant set, like the Region of Attraction, and for all $t \ge 0$, can be shown to hold under minor smoothness assumption on $f$, see~\cite{Khalil_1996}.
\vspace{-0.35cm} \paragraph{Stability of Nonlinear ODEs}
We now use the solution map of the ODE~\eqref{eqn: ODE} to define several notions of stability.
   \begin{defn}  \label{defn: asym and exp stab}
	 {      The equilibrium point $x=0$ of ODE~\eqref{eqn: ODE} is,
		\begin{itemize}
			\item Stable if, for each $\eps>0$, there exists $\delta>0$ such that 
			\vspace{-0.25cm} \begin{align} \nonumber
				||\phi_f(x,t)||_2<\eps \text{ for all } x \in B_\delta(0) \text{ and } t \ge 0.  \end{align}
			\item Asymptotically stable if it is stable and there exists $\delta>0$ such that $\lim_{t \to \infty} ||\phi_f(x,t)||_2=0 \text{ for all } x \in B_\delta(0).$
			%   \begin{align} \nonumber
				%         \lim_{t \to \infty} ||\phi_f(x,t)||_2=0 \text{ for all } x \in B_\delta(0).
				%   \end{align}
			\item Exponentially stable if there exists $\lambda,\mu>0$ such that 
			\vspace{-0.2cm} \begin{align} \nonumber
				||\phi_f(x,t)||_2<\mu e^{-\lambda t} ||x||_2 \text{ for all } x \in B_\delta(0) \text{ and } t \ge 0.
			\end{align}
	\end{itemize}}
\end{defn}
%\begin{defn} \label{defn: asym and exp stab}
%	We say the set $U \subset \R^n$ is an asymptotically stable set of the ODE~\eqref{eqn: ODE} if:
%	\begin{enumerate}
%		\item $U$ contains a neighbourhood of the origin.
%		\item For any $x \in U$ we have that $\phi_f(x,t) \in U$ for all $t \in [0,\infty)$ and $\lim_{t \to \infty} ||\phi_f(x,t)||_2=0$.
%	\end{enumerate}
%	Furthermore, if there also exists $\delta, \mu>0$ such that for any $x \in U$ we have that
%	\begin{align} \label{eqn: exp stab}
%		||\phi_f(x,t)||_2 \le \mu e^{-\delta t} ||x||_2 \text{ for all } t \ge 0,
%	\end{align}
%	then we say $U\subset \R^n$ is an exponentially stable set of the ODE~\eqref{eqn: ODE}.
%\end{defn}
 {Given an ODE, if origin is an asymptotically stable equilibrium point of the ODE then we will say that the ODE is \textbf{locally asymptotically stable}. }

%We next define the Region of Attraction (ROA) of an ODE that can be thought of as the intersection of all asymptotically stable sets.
%\begin{defn} \label{defn: ROA}
%	The (Maximal) Region of Attraction (ROA) of the ODE~\eqref{eqn: ODE} is defined as the following set:
%	\begin{align} \label{eqn: ROA}
%		ROA_f:= \{x \in \R^n: \lim_{t \to \infty} ||\phi_f(x,t)||_2=0\}.
%	\end{align}
%\end{defn}
%The ROA of the ODE~\eqref{eqn: ODE} can be thought of as the ``maximal" asymptotically stable set. That is if $U \subset \R^n$ is an asymptotically stable set of the ODE~\eqref{eqn: ODE} then $U \subseteq ROA_f$. 

\paragraph{Certifying the Stability of Nonlinear ODEs}
 {In general there is no analytical expression for the solution map of a nonlinear ODE. Hence, directly certifying whether a nonlinear ODE is locally asymptotically stable by first finding the solution map, $\phi_f,$ and then showing $\lim_{t \to \infty} ||\phi_f(x,t)||_2 =0$ over some set $B_\delta(0)$ is challenging. Fortunately, there exists several methods that can certify the local asymptotic stability of an ODE without first finding the solution map. Arguably, the most important of these methods, that we now state next, are Lyapunov's first and second methods.}

%\paragraph{Lyapunov's First and Second Methods}
\begin{lem}[Lyapunov's First Method] \label{lem: LF first method}
	Consider an ODE~\eqref{eqn: ODE} defined by some vector field $f \in C^1(\R^n, \R^n)$ with $f(0)=0$. Let
\vspace{-0.55cm}	\begin{align} \label{eq: lineraization of f marix}
		A:=\begin{bmatrix} \frac{\partial f_1(0)}{\partial x_1} & \cdots  & \frac{\partial f_1(0)}{\partial x_n} \\
			\vdots & \ddots & \vdots\\
			\frac{\partial f_n(0)}{\partial x_1} & \cdots & \frac{\partial f_n(0)}{\partial x_n} \end{bmatrix} \in \R^{n \times n}.
	\end{align}
	It follows that,
	\begin{itemize}
		\item If all the real parts of the eigenvalues of $A$ are negative then the ODE is locally asymptotically stable.
		\item If there exists an eigenvalue of $A$ whose real part is positive then the ODE is \textbf{not} locally asymptotically stable. 
	\end{itemize}
\end{lem}

From Lem.~\ref{lem: LF first method} we see that in the case when $A \in \R^{n \times n}$, given in Eq.~\eqref{eq: lineraization of f marix}, has an eigenvalue that  {is purely imaginative} we are unable to use Lyapunov's first method to certify whether the associated ODE is locally asymptotically stable or not. For this case we can still certify local asymptotic stability using Lyapunov's Second Method, stated next.

%\begin{lem}[Lyapunov's Second Method] \label{lem: LF subset of ROA}
%	Consider an ODE~\eqref{eqn: ODE} defined by some vector field $f \in C^1(\R^n, \R^n)$ with $f(0)=0$. Suppose there exits $V \in C^1(\R^n, \R)$ and a set $\Omega \subseteq \R^n$ with $0 \in \Omega$ such that
%	\begin{align} \label{eq: LF ineq 1}
%		& V(0) =0 \text{ and } V(x)   > 0 \text{ for all } x \in \Omega/\{0\}, \\  \label{eq: LF ineq 2}
%		& \nabla V(x)^\top f(x)  < 0 \text{ for all } x \in \Omega/\{0\}.
%	\end{align}
%	If $\gamma>0$ is such that $\{x\in \R^n: V(x) \le  \gamma\} \subseteq  \Omega$ then  {the ODE is locally asymptotically stable.} \textcolor{red}{[$\{x\in \R^n: V(x) \le \gamma\}$ is an asymptotically stable set (Def.~\ref{defn: asym and exp stab})]}.
%\end{lem}
%\paragraph{Converse Lyapunov Theory}
% {We next recall the converse of Lem.~\ref{lem: LF subset of ROA} that provides the necessary conditions for the existence of such a function $V \in C^1(\R^n, \R)$ satisfying Eqs~\eqref{eq: LF ineq 1} and~\eqref{eq: LF ineq 2}.}
\begin{thm}[Lyapunov's  Second Method~\cite{bacciotti2005liapunov}] \label{thm: converse LF}
Consider an ODE~\eqref{eqn: ODE} defined by some $f \in C^1(\R^n, \R^n)$ with $f(0)=0$. The ODE is locally asymptomatically stable if and only if there exists $R>0$ and $V\in C^1(\R^n, \R^n)$ that satisfies,
\vspace{-0.25cm} \begin{align} \label{eq: LF ineq 1}
&V(0)=0, \quad V(x)>0 \text{ for all } x \in {B_R(0)/\{0\}}, \\ \label{eq: LF ineq 2}
& \nabla V(x)^\top f(x) <0 \text{ for all } x \in B_R(0)/\{0\}.
\end{align}
\end{thm}
 {In the special case of asymptotically stable linear systems, $\dot{x}(t)=Ax(t)$, it is well known that there exists a quadratic LF, $V(x)=x^\top Px$ where $P>0$, and the Lyapunov condition of Thm.~\ref{thm: converse LF} reduces to the Matrix Equation~\eqref{eq: Lyapunov eq for linear systems} as shown in the next theorem.}
%
 %Theorem~\ref{thm: LF linear systems}, stated next, can be used to show that there exists quadratic Lyapunov equation if and only if the Matrix Equation~\eqref{eq: Lyapunov eq for linear systems} has a positive semidefinite matrix solution.
\begin{thm}[\cite{williams2007linear}] \label{thm: LF linear systems}
For any symmetric positive definite matrix $Q \in \R^{n \times n}$, the Lyapunov matrix equation,
\vspace{-0.2cm} \begin{align} \label{eq: Lyapunov eq for linear systems}
	A^\top P + P A=-Q,
\end{align}
has a unique symmetric positive definite solution $P \in \R^{n \times n}$ if every eigenvalue of $A \in \R^{n \times n}$ has strictly negative real part.
\end{thm}

%Theorem~\ref{thm: converse LF} shows that there exists a differentiable LF for locally asymptotically stable ODEs. Furthermore, in the special case of asymptotically stable linear systems, Theorem~\ref{thm: LF linear systems} can be used to show that there exists a quadratic LF. The goal of this paper is to bridge these two results by proposing conditions for which there exists a partially quadratic LF, a LF that is quadratic in a subset of the state variables. We achieve this goal later in Section~\ref{sec: converse partiall quadratic LF} Prop.~\ref{prop: seperable LF}, showing that there exists a partially quadratic LF based on the existence of a center manifold. Before we state and prove Prop.~\ref{prop: seperable LF} we first recall the necessary center manifold background theory in the next section.  

\paragraph{Coordinate changes for block diagonalization of linearization matrix} In order to state the main result in Thm.~\ref{thm: seperable LF}, that there exists a converse partially quadratic LF, we must first make a coordinate change to the ODE~\eqref{eqn: ODE}. This coordinate change will allow us to write the ODE as two coupled ODEs whose state variables will either appear quadratically or non-quadratically in our converse LF. 

Since we are concerned with certifying whether the ODE~\eqref{eqn: ODE} is locally stable, WLOG, we now assume that the associated linearization matrix, $A \in \R^{n \times n}$, given in Eq.~\eqref{eq: lineraization of f marix}, has $k \in \N$ purely imaginary eigenvalues and that the remaining eigenvalues of $A$ have negative real parts. We assume this WLOG because in the case where all of the eigenvalues of $A$ have negative real parts (i.e $k=0$) we can certify that the ODE is locally stable by Lem.~\ref{lem: LF first method}. Alternatively, if any of the eigenvalues of $A$ have positive real part then by Lem.~\ref{lem: LF first method} we can certify that the ODE is not locally asymptotically stable. In both of these cases there would be no need to find a LF.

Now, for a matrix, $A \in \R^{n \times n}$, that has eigenvalues that are either purely imaginary or have negative real parts, Lemma~\ref{lem: block diagonal matrix} (found in the Appendix) shows that there exists an invertible matrix $T \in \R^{n \times n}$ for which
\vspace{-0.4cm} \begin{align} \label{eq: block digaonal form}
	\text{ } \quad \qquad  TAT^{-1}= \begin{bmatrix} A_1 &   0 \\ 0   & A_2 \end{bmatrix} \in \R^{n \times n},
\end{align}
where $A_1$ has only purely imaginary eigenvalues and $A_2$ has eigenvalues with only negative real part.

Note that for any vector field $f$ with associated linearization matrix, $A \in \R^{n \times n}$, given in Eq.~\eqref{eq: lineraization of f marix}, it follows that $f(x)= Ax + \tilde{g}(x)$, where $\tilde{g}(x):= f(x) - Ax$ is such that $\frac{\partial}{\partial x_i} \tilde{g}(0)=0$.  Thus given an ODE~\eqref{eqn: ODE}, defined by a vector field $f$, WLOG we assume ${f}(x)= Ax + \tilde{g}(x)$ for some function $\tilde{g}$ such that $\frac{\partial}{\partial x_i} \tilde{g}(0)=0$. Then, by making the coordinate change $\begin{bmatrix}
	z_1 \\z_2
\end{bmatrix}=Tx$ to ODE~\eqref{eqn: ODE} we can consider the equivalent nonlinear ODE:
\begin{align} \label{ODE: zero real part}
	\dot{z_1}(t)   & =  A_1 z_1(t) + g_1(z_1(t),z_2(t))\\ \label{ODE: negative real part}
	\dot{z_2}(t) &= A_2 z_2(t) + g_2(z_1(t),z_2(t)),
\end{align}
where $A_1 \in \R^{k \times k}$ has purely imaginary eigenvalues, $A_2 \in \R^{(n-k) \times (n-k)}$ has eigenvalues with only negative real part, $g_1\in C^1( \R^k \times \R^{n-k} ,\R^k)$ is such that $\frac{\partial}{\partial x_i} g_1(0)=0$ for $i \in \{1,\dots,n\}$,  $g_2 \in C^1(\R^k \times \R^{n-k}, \R^{n-k})$ is such that $\frac{\partial}{\partial x_i} g_2(0)=0$ for $i \in \{1,\dots,n\}$, $k=\sum_{\lambda \in S}Dim \bigg(Ker(\lambda I - A)^{m(\lambda)} \bigg)$, $m(\lambda)$ is the algebraic multiplicity of eigenvalue $\lambda$, and $S \subset \mathbb{C}$ is the set of distinct eigenvalues of $A$ with zero real part.

%Prop.~\ref{prop: Stability of reduced system} shows that we can certify the stability of an ODE given by  Eqs.~\eqref{ODE: zero real part} and~\eqref{ODE: negative real part} by certifying the stability of an ODE with smaller state-space dimension, ODE~\eqref{ODE: reduced}. We would like to use Prop.~\ref{prop: Stability of reduced system}  to develop efficient numerical methods that certify the local asymptotic stability of high state space dimensional ODEs by solving lower dimensional problems. However, Prop.~\eqref{prop: Stability of reduced system} can not be directly used in this way since we do not have an analytical expression for the vector field of ODE~\eqref{ODE: reduced}; $\eta$ is characterized as the solution of nonlinear PDE~\eqref{PDE: eta} (which is challenging to solve). Therefore, in the next section we take an alternative approach by using Theorem~\ref{thm: The Center Manifold Theorem} and Prop.~\ref{prop: Stability of reduced system} to show the existence of a converse LF with a particular structure that can be efficiently searched for.

\vspace{-0.2cm}\section{Converse Partially Quadratic LFs} \vspace{-0cm} \label{sec: converse partiall quadratic LF}

 {We now use the Center Manifold Theorem (Thm.~\ref{thm: The Center Manifold Theorem} found in the Appendix) to prove the main result of the paper, Thm.~\ref{thm: seperable LF}, that shows that under certain conditions there exists a partially quadratic LF. Before stating Thm.~\ref{thm: seperable LF} we first give a preliminary result. This preliminary result shows that the conditions required in our main result are satisfied by some commonly encountered systems. Note that similar conditions appear in~\cite{atassi1999separation}}.

\begin{lem} \label{lem: conditions which ass holds}
 Consider an ODE~\eqref{eqn: ODE} defined by a vector field $f$. Suppose one or more of the following statements holds:
 % given by Eqs.~\eqref{ODE: zero real part} and~\eqref{ODE: negative real part} and their associated reduced ODE~\eqref{ODE: reduced}, defined by some vector field $f(y):=A_1 y+ g_1(y, \eta(x))$, where  $A_1 \in \R^{k \times k}$, $g_1 : \R^k \times \R^{(n-k)} \to \R^k$ and $\eta: \R^k \to \R^{(n-k)}$. Suppose one or more of the following statements holds:
%\begin{itemize}
%	\item ODE~\eqref{ODE: reduced} is locally Exponentially stable (Def.~\ref{defn: asym and exp stab}).
%	\item ODE~\eqref{ODE: reduced} is a gradient system. That is its vector field is of the form $f(y)=-\nabla V(y)$, where $V: \R^n \to \R$  is some function that satisfies $\nabla V(0)=0$, $V(y) \ge 0$ for all $x \in \R^n$ and $V(y)=0$ iff $x=0$.
%	\item ODE~\eqref{ODE: reduced} is locally asymptotically stable and one dimensional, that is $k=1$.
%\end{itemize}
\begin{itemize}
	\item The ODE is locally exponentially stable (Def.~\ref{defn: asym and exp stab}).
	\item The ODE is a gradient system. That is its vector field is of the form $f(y)=-\nabla V(y)$, where $V: \R^n \to \R$  is some function that satisfies $\nabla V(0)=0$, $V(y) \ge 0$ for all $x \in \R^n$ and $V(y)=0$ if and only if $x=0$.
	\item The ODE is locally asymptotically stable with a one dimensional state space.
\end{itemize}
Then there exists a radius, $R>0$, and a LF, $W \in C^1(\R^k, \R)$, satisfying
\vspace{-0.1cm}\begin{align} \label{ass: W of reduced ODE}
	&	W(y) \ge 0 \text{ for all } y \in B_R(0),\\ \nonumber
	&	W(y) =0 \text{ if and only if } y=0,\\ \nonumber
	&	\nabla W(y) \hspace{-0.05cm}  ^\top \hspace{-0.05cm}  f(y)   < \hspace{-0.1cm} -c_1 \alpha(||y||_2)^2  \hspace{-0.05cm}  \text{ for all } \hspace{-0.05cm}  y \hspace{-0.05cm}  \in  \hspace{-0.05cm}  B_R(0),\\ \nonumber
	&	||\nabla W(y)||_2 < c_2 \alpha(||y||_2)  \text{ for all } y \in B_R(0),
\end{align}
%where $c_1,c_2 \in [0,\infty)$, $\eta$ solves the PDE~\eqref{PDE: eta} (known to exist by Thm.~\ref{thm: The Center Manifold Theorem}), and $\alpha: [0, \infty) \to [0, \infty)$ is such that $\alpha(0)=0$.
where $c_1,c_2 \in [0,\infty)$ and $\alpha: [0, \infty) \to [0, \infty)$ is such that $\alpha(0)=0$. 
\end{lem}
\begin{proof}
	Suppose the ODE is locally exponentially stable. Then by Corollary~77 from Page~245 in~\cite{vidyasagar2002nonlinear} there exists a LF that satisfies Eq.~\eqref{ass: W of reduced ODE} with $\alpha(y):=y$.
	
	Next, suppose the ODE is a gradient system. Then $W(y):=V(y)$ satisfies Eq.~\eqref{ass: W of reduced ODE} with $\alpha(y):=||\nabla V(y)||_2$.
	
	Finally, suppose the ODE has state space dimension equal to one. By defining $V(y)=-\int_0^y f(x) dx$ we see that the ODE is a gradient system. Hence, Eq.~\eqref{ass: W of reduced ODE} is satisfied.
\end{proof}

%Therefore, in order for there to exist a partially quadratic LF, as it is shown in Prop.~\ref{prop: seperable LF}, the ODE given in Eqs.~\eqref{ODE: zero real part} and~\eqref{ODE: negative real part} must be asymptotically stable.

%where $A_1 \in \R^{k \times k}$ has eigenvalues with only zero real part, $A_2 \in \R^{(n-k) \times (n-k)}$ has eigenvalues with only strictly negative real part, $g_1: \R^k \times \R^{n-k} \to \R^k$ is such that $\frac{\partial}{\partial x_i} g_1(0)=0$ for $i \in \{1,\dots,n\}$ and  $g_2: \R^k \times \R^{n-k} \to \R^{n-k}$ is such that $\frac{\partial}{\partial x_i} g_2(0)=0$ for $i \in \{1,\dots,n\}$. Suppose Assumption~\ref{ass: partially quadtic LF} holds

We now show that for the ODE given in Eqs.~\eqref{ODE: zero real part} and~\eqref{ODE: negative real part} if Eq.~\eqref{ass: W of reduced ODE} holds for the associated reduced ODE~\eqref{ODE: reduced} then there exists a partially quadratic LF of the form given in Eq.~\eqref{eq: Form of V} that can certify the local asymptotic stability of the ODE.
\begin{thm}[{Existence of converse partially quadratic LFs}] \label{thm: seperable LF}
 {Consider an ODE given by Eqs.~\eqref{ODE: zero real part} and~\eqref{ODE: negative real part} and the associated reduced ODE~\eqref{ODE: reduced}. Suppose there exists a function $W \in C^1(\R^k, \R)$ satisfying Eq.~\eqref{ass: W of reduced ODE} for the vector field $f$ of the reduced ODE~\eqref{ODE: reduced}. Then, the ODE given by Eqs.~\eqref{ODE: zero real part} and~\eqref{ODE: negative real part} is locally asymptotically stable if and only if there exists a matrix $P>0$, a scalar $R>0$ and functions $J \in C^1 (\R^k , \R)$ and $H \in C^1 (\R^k , \R^{(n-k)})$ such that
\begin{align} \label{LFineq}
PA_2 + A_2^\top P =-I,
\end{align}
and
\vspace{-0.2cm} \begin{align} \label{V: greater than 0 away from the origin}
	&	V(z_1,z_2) > 0 \text{ for all } (z_1,z_2) \in B_R (0)/ \{0\},\\ \label{V:  0 at 0}
	&	V(0,0) =0 , \\ \label{V: decreasing along traj}
	&	\nabla V(z_1,z_2)^\top \begin{bmatrix}
		A_1 z_1 + g_1(z_1,z_2)\\
		A_2 z_2 + g_2(z_1,z_2)
	\end{bmatrix} < 0 \\ \nonumber
	& \hspace{3cm} \text{ for all } (z_1,z_2) \in B_R(0)/\{0\}.
\end{align}
\vspace{-0.75cm}\begin{align} \label{eq: Form of V}
\text{where }	V(z_1,z_2):= J(z_1) + z_2^\top H(z_1) + z_2^\top P z_2.
\end{align} }
%where $PA_2 + A_2^\top P =-I$, $J \in C^1 (\R^k , \R)$ and $H \in C^1 (\R^k , \R^{(n-k)})$, such that, 
%\begin{align} \label{V: greater than 0 away from the origin}
%&	V(z_1,z_2) > 0 \text{ for all } (z_1,z_2) \in B_R (0)/ \{0\},\\ \label{V:  0 at 0}
%&	V(0,0) =0 , \\ \label{V: decreasing along traj}
%&	\nabla V(z_1,z_2)^\top \begin{bmatrix}
%		A_1 z_1 + g_1(z_1,z_2)\\
%		A_2 z_2 + g_2(z_1,z_2)
%	\end{bmatrix} < 0 \\ \nonumber
%& \hspace{3cm} \text{ for all } (z_1,z_2) \in B_R(0)/\{0\}.
%\end{align}
\end{thm}
%\textcolor{red}{for local stability dont we need a subset inside $B_\delta(0)$?}
\vspace{-0.1cm} \begin{proof}
	 {First suppose that there exists a matrix $P>0$, a scalar $R>0$ and functions $J \in C^1 (\R^k , \R)$ and $H \in C^1 (\R^k , \R^{(n-k)})$ such that Eqs.~\eqref{LFineq}, \eqref{V: greater than 0 away from the origin}, \eqref{V:  0 at 0} and~\eqref{V: decreasing along traj} hold, where $V$ is given by Eq.~\eqref{eq: Form of V}. Now, it follows that $V$ is a LF for the ODE given by Eqs.~\eqref{ODE: zero real part} and~\eqref{ODE: negative real part} and hence this ODE is locally asymptotically stable by Thm.~\ref{thm: converse LF}.}
	
 {On the other hand let us now suppose the ODE given by  Eqs.~\eqref{ODE: zero real part} and~\eqref{ODE: negative real part} is locally asymptotically stable.} Consider the following function,
\vspace{-0.2cm}	\begin{align*}
V(z_1,z_2)= W(z_1) + (z_2 - \eta(z_1))^\top P (z_2 - \eta(z_1)),
	\end{align*}
where $W$ satisfies Eq.~\eqref{ass: W of reduced ODE} for some radius $R_1>0$ and for the vector field of the reduced ODE~\eqref{ODE: reduced}, given by $f(y)=A_1 y(t) + g_1(y(t), \eta(y(t)))$, where $\eta$  satisfies PDE~\eqref{PDE: eta} for some radius $R_2>0$ (known to exist by Thm.~\ref{thm: The Center Manifold Theorem}), and $P>0$ is such that 
\vspace{-0.4cm}  \begin{align} \label{pfeq: LMI}
	PA_2 + A_2^\top P =-I.
\end{align}
Note that such a $P>0$ exists by Thm.~\ref{thm: LF linear systems} since $A_2$ is defined in Eq.~\eqref{ODE: negative real part} to have eigenvalues with only negative real part.

Now, it clearly follows by multiplying out the quadratic terms in $V$ that,
	\begin{align} \label{pfeq: V}
	V(z_1,z_2) \hspace{-0.05cm} = \hspace{-0.05cm}  W(z_1)  \hspace{-0.05cm}  + \hspace{-0.05cm}  \eta(z_1)^\top \hspace{-0.05cm}  P \eta(z_1) \hspace{-0.05cm}   - \hspace{-0.05cm}  2 z_2^\top P \eta(z_1) \hspace{-0.05cm}  + \hspace{-0.05cm}  z_2^\top P z_2.
	\end{align}
	Hence, $V$ satisfies Eq.~\eqref{eq: Form of V} with $J(z_1)=W(z_1) + \eta(z_1)^\top P \eta(z_1)$ and $H(z_1)= -2P\eta(z_1)$.
	
	We next show that $V$ satisfies Eqs.~\eqref{V: greater than 0 away from the origin} and~\eqref{V:  0 at 0}. The function $V$ comprises of the sum of two positive terms and thus it is clear $V(z_1,z_2) \ge 0$ for all $(z_1,z_2) \in B_\delta(0)$. Clearly $V(z_1,z_2)=0$ if and only if both of these positive terms are zero. Now, $W(z_1)=0$ if and only if $z_1=0$ and $(z_2 - \eta(z_1))^\top P (z_2 - \eta(z_1))=0$ if and only if $z_2=\eta(z_1)$. If $z_1=0$ and $z_2=\eta(z_1)$ then $z_2=\eta(0)=0$ (note that $\eta(0)=0$ by Theorem~\ref{thm: The Center Manifold Theorem}). Therefore $V(z_1,z_2)=0$ if and only if  $(z_1,z_2)=0$.
	
We next show that $V$ satisfies Eq.~\eqref{V: decreasing along traj}. First note that $g_1$ and $g_2$ defined in Eqs.~\eqref{ODE: zero real part} and~\eqref{ODE: negative real part} are such that $\nabla g_1(0,0)=0$ and $\nabla g_2(0,0)=0$. Then by Lem.~\ref{lem: Lip bound} (found in the Appendix) it follows that for $\eps:=\frac{1}{2}\min\{\frac{2 c_1}{c_2^2}, \frac{2}{1+ 4 \lambda_{max}}  \}>0 $, where $\lambda_{max}>0$ is the largest eigenvalue of $P>0$, there exists $R_3>0$ such that 
\begin{align} \label{pfeq: g1 g2 lip}
&	||g_1(u_1,u_2) - g_1(v_1,v_2) ||_2< \eps ||(u_1,u_2)- (v_1,v_2)||_2\\ \nonumber
	&	||g_2(u_1,u_2) - g_2(v_1,v_2) ||_2< \eps ||(u_1,u_2)- (v_1,v_2)||_2\\ \nonumber
		& \hspace{2cm} \text{ for all } (u_1,u_2),(u_1,u_2) \in B_{R_3}(0).
\end{align}

It now follows from application of Eq.~\eqref{pfeq: g1 g2 lip} that,
\vspace{-0.2cm} 	\begin{align} \label{pfeq: Vdot less than 0}
	&	\nabla V(z_1,z_2)^\top \begin{bmatrix}
			A_1 z_1 + g_1(z_1,z_2)\\
			A_2 z_2 + g_2(z_1,z_2)
			\end{bmatrix}\\ \nonumber
		&=\nabla W(z_1)^\top (A_1 z_1 + g_1(z_1,z_2)) \\  \nonumber
		& \quad -2 (z_2 - \eta(z_1))^\top P \nabla \eta (z_1)^\top (A_1 z_1 + g_1(z_1,z_2) )  \\  \nonumber
		& \quad + 2 (z_2 - \eta(z_1))^\top P (A_2 z_2 + g_2(z_1,z_2) )\\  \nonumber
		&= \nabla W(z_1)^\top (A_1 z_1 + g_1(z_1,\eta(z_1))) \\  \nonumber
		& \quad + \nabla W(z_1)^\top (g_1(z_1,z_2) - g_1(z_1,\eta(z_1))) \\  \nonumber
			& \quad -2 (z_2 - \eta(z_1))^\top P (A_2 \eta (z_1) + g_2(z_1,\eta(z_1))) \\  \nonumber
			& \quad + (z_2-\eta(z_1))^\top (PA_2 + A_2^\top P) (z_2-\eta(z_1)) \\  \nonumber
			& \quad + 2(z_2-\eta(z_1))^\top  P (A_2 \eta(z_1)  + g_2(z_1,z_2))\\  \nonumber
			& < \hspace{-0.05cm} -c_1 \alpha(||z_1||_2)^2 \hspace{-0.05cm} + \hspace{-0.05cm} ||\nabla \hspace{-0.02cm} W(\hspace{-0.05cm} z_1 \hspace{-0.05cm} )||_2 \hspace{-0.04cm} ||g_1(z_1,z_2) \hspace{-0.05cm} - \hspace{-0.05cm}g_1(z_1,\eta( \hspace{-0.05cm} z_1 \hspace{-0.05cm} ))||_2\\  \nonumber
			& \quad \hspace{-0.2cm}  - \hspace{-0.1cm} ||z_2 \hspace{-0.1cm} - \hspace{-0.1cm} \eta( \hspace{-0.05cm} z_1 \hspace{-0.05cm} )||_2^2 \hspace{-0.05cm} + \hspace{-0.05cm} 2(z_2 \hspace{-0.1cm} - \hspace{-0.1cm} \eta(z_1))^\top \hspace{-0.1cm}  P  (g_2(z_1 \hspace{-0.05cm} ,z_2) \hspace{-0.05cm} - \hspace{-0.05cm} g_2(z_1, \hspace{-0.05cm} \eta( \hspace{-0.05cm} z_1 \hspace{-0.05cm} )))\\  \nonumber
			& \le \hspace{-0.05cm}  - \hspace{-0.05cm} c_1 \alpha(||z_1||_2)^2 \hspace{-0.05cm} + \hspace{-0.05cm}  \bigg( \hspace{-0.1cm} c_2 \sqrt{\eps} \alpha(||z_1||_2) \hspace{-0.1cm}  \bigg) \bigg(\sqrt{\eps} || z_2 - \eta(z_1)||_2\bigg)\\  \nonumber
			& \quad   -  ||z_2-\eta(z_1)||_2^2 +2\eps \lambda_{max}||z_2-\eta(z_1)||_2^2\\  \nonumber
			& \le\hspace{-0.1cm}  -  \hspace{-0.1cm}  \left( \hspace{-0.1cm}  c_1 \hspace{-0.1cm} - \hspace{-0.05cm}  \frac{c_2^2 \eps}{2} \hspace{-0.1cm}  \right) \hspace{-0.1cm}   \alpha(||z_1||_2)^2 \hspace{-0.1cm} - \hspace{-0.1cm}  \left( \hspace{-0.1cm} 1 \hspace{-0.05cm} -  \hspace{-0.05cm} \frac{\eps(1+4 \lambda_{max})}{2} \hspace{-0.1cm} \right) \hspace{-0.1cm} ||z_2 \hspace{-0.1cm} - \hspace{-0.05cm} \eta( \hspace{-0.05cm} z_1 \hspace{-0.05cm} )||_2^2 \\  \nonumber
			&<0 \text{ for all } (z_1,z_2) \in B_R(0)/\{0\},
	\end{align}
where $R=\min\{R_1,R_2,R_3\}>0$ and $\lambda_{max}>0$ is the largest eigenvalue of $P>0$.

The second equality from Eq.~\eqref{pfeq: Vdot less than 0} follows from the application of PDE~\eqref{PDE: eta}, found in Thm.~\ref{thm: The Center Manifold Theorem} from the Appendix. The first inequality of Eq.~\eqref{pfeq: Vdot less than 0} follows from Eqs.~\eqref{ass: W of reduced ODE} and~\eqref{pfeq: LMI} and the Cauchy Schwarz inequality. The second inequality of Eq.~\eqref{pfeq: Vdot less than 0} follows from Eq.~\eqref{pfeq: g1 g2 lip}. The third inequality of Eq.~\eqref{pfeq: Vdot less than 0} using the inequality $xy \le \frac{x^2 + y^2}{2}$ for all $x,y \in \R$. The fourth and final inequality in Eq.~\eqref{pfeq: Vdot less than 0} follows since $\eps:=\frac{1}{2}\min\{\frac{2 c_1}{c_2^2}, \frac{2}{1+ 4 \lambda_{max}}  \}$ and hence $c_1 - \frac{c_2^2 \eps}{2}>0$ and $1 - \frac{\eps(1+4 \lambda_{max})}{2}>0$.

Therefore it follows that $V$, given in Eq.~\eqref{pfeq: V}, satisfies Eqs.~\eqref{V: greater than 0 away from the origin},~\eqref{V:  0 at 0} and~\eqref{V: decreasing along traj} for $R\hspace{-0.1cm} = \hspace{-0.1cm}  \min\{R_1,R_2,R_3\} \hspace{-0.1cm} > \hspace{-0.1cm} 0$.
	\end{proof}
 {If there exists a function $W$ satisfying Eq.~\eqref{ass: W of reduced ODE} for the reduced ODE~\eqref{ODE: reduced} then Thm.~\ref{thm: seperable LF} shows that the ODE given in Eqs.~\eqref{ODE: zero real part} and~\eqref{ODE: negative real part} is locally asymptotically stable if and only if there exists a partially quadratic LF. Further to this Lem.~\ref{lem: conditions which ass holds} provides some sufficient conditions that guarantee the existence of such a function $W$. In the next corollary we combine these results to show the existence of partially quadratic LFs for systems whose linearization matrix has only one purely imaginary eigenvalue. This corollary provides the theoretical justification for the search of partially quadratic LFs to certify local asymptotic stability in our numerical examples in Sec.~\ref{sec: numerical examples}.   }
\begin{cor} \label{cor: only one imaginary eigenvalue implies partiall quad LF}
 {Consider an ODE~\eqref{eqn: ODE} with associated linearization matrix $A \in \R^{n \times n}$ defined in Eq.~\eqref{eq: lineraization of f marix}. Suppose there is a single purely imaginary eigenvalue of $A$ and all other eigenvalues of $A$ have negative real parts. Then the ODE is stable if and only if there exists a partially quadratic LF of the form given in Eq.~\eqref{eq: Form of V} that satisfies Eqs~\eqref{V: greater than 0 away from the origin}, \eqref{V:  0 at 0} and~\eqref{V: decreasing along traj}.}
\end{cor}
\begin{proof}
		 {Follows using Prop.~\ref{prop: Stability of reduced system}, Lem.~\ref{lem: conditions which ass holds} and Thm.~\ref{thm: converse LF}}.
\end{proof}
 {Note, the proof of the existence of partially quadratic LFs given in Thm.~\ref{thm: seperable LF} is non-constructive, being based on the center manifold, $z_2=\eta(z_1)$, for which in general there is no analytical formula. In the special case where the center manifold is analytically known the proof becomes constructive. In the following illustrative example we use Eq.~\eqref{pfeq: V} to construct a partially quadratic LF without any computation. Later, in Sec.~\ref{sec: SOS problems} we will consider high dimensional systems for which the center manifold is not known and hence such a LF cannot be constructively found. For such systems we will use numerical methods to search for partially quadratic LFs to certify local asymptotic stability.}

%However, the knowledge that there exists partially quadratic LFs is useful because later in Section~\ref{sec: SOS problems} it allows us to propose efficient numerical algorithms to search for such partially quadratic LFs. 

\paragraph{An Illustrative Example}  Consider the following ODE,
\vspace{-0.7cm}  \begin{align} \label{ODE: illistrative example}
	\dot{x}_1(t) & = -x_1(t)x_2(t)\\ \nonumber
	\dot{x}_2(t) & = - x_2(t) + x_1(t)^2 - 2 x_2(t)^2.
\end{align}
The associated linearization matrix, found in Eq.~\eqref{eq: lineraization of f marix}, is given by $A=\begin{bmatrix} 0 & 0 \\ 0 & -1 \end{bmatrix}$. The matrix $A \in \R^{2 \times 2}$ has eigenvalues $0$ and $-1$ and thus ODE~\eqref{ODE: illistrative example} cannot be certified as locally asymptotically stable using Lyapunov's first method (Lem.~\ref{lem: LF first method}). 

Note that without any coordinate transformations ODE~\eqref{ODE: illistrative example} is already in the form of ODE given by Eqs.~\eqref{ODE: zero real part} and~\eqref{ODE: negative real part} with $A_1=0$ and $A_2=-1$. Thus setting $P=0.5>0$ it follows that $PA_2 + A_2^\top P =-I$. It is shown in~\cite{roberts1997low} that $\eta(y)=y^2$ gives the center manifold. Hence, the reduced ODE~\eqref{ODE: reduced}  associated with ODE~\eqref{ODE: illistrative example} is,
\vspace{-0.15cm} \begin{align} \label{ODE: illistrative example reduced}
	\dot{y}(t) & = -y(t)^3.
\end{align}
ODE~\eqref{ODE: illistrative example reduced} has a one dimensional state space so by Lem.~\ref{lem: conditions which ass holds} it follows that there exists a function $W$ satisfying Eq.~\eqref{ass: W of reduced ODE}. Specifically, if we let $W(y)= \frac{y^4}{4}$ it can be shown $W$ satisfies Eq.~\eqref{ass: W of reduced ODE} with $\alpha(y)=y^3$. The proof of Thm.~\ref{thm: seperable LF} then shows that the function given in Eq.~\eqref{pfeq: V} is a LF. For this ODE this then implies that the following function is a LF to ODE~\eqref{ODE: illistrative example},
\vspace{-0.3cm} \begin{align} \label{V: illistrative example}
V(x_1,x_2)= x_1^4 /4 + 0.5 (x_2 - x_1^2)^2.
\end{align}
Clearly, the LF given in Eq.~\eqref{V: illistrative example} is partially quadratic since the $x_2$ terms appear with degree at most $2$ while the $x_1$ terms can have degree greater $2$. We have plotted the largest set of initial conditions that this LF can certify as asymptotically stable as the green region in Fig.~\ref{fig: Illistrative example}. %The figure also demonstrates how trajectories of the ODE tend towards the center manifold before converging to the equilibrium point at the origin. 

\begin{figure*}[h!]%[t!] \centering
	\begin{subfigure}[t]{ 0.333 \textwidth} 
		\includegraphics[width=\linewidth, trim = {1.75cm 1cm 1cm 1cm}, clip]{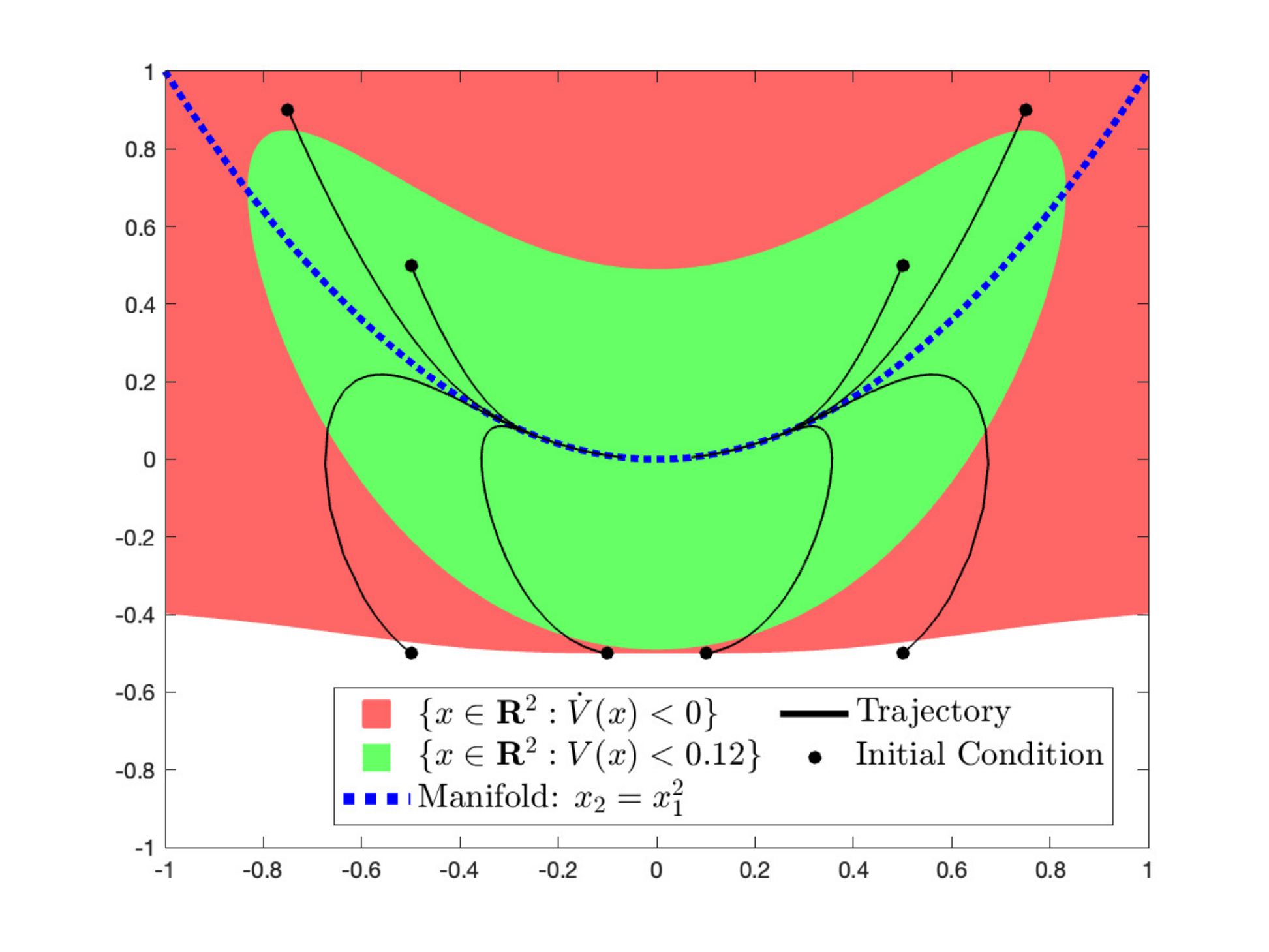}
	%	\vspace{-10pt}
	\caption{}
		%\caption{\footnotesize Graph showing that the LF given in Eq.~\eqref{V: illistrative example} certifies the local stability of ODE~\eqref{ODE: illistrative example} according to Lem.~\ref{lem: LF subset of ROA}. The center manifold, $y=x^2$, is also plotted as the dotted blue line. Several trajectories of the ODE for various initial conditions are plotted as the black curves. }
\label{fig: Illistrative example}
	\end{subfigure}%
	\begin{subfigure}[t]{0.333 \textwidth}
		\includegraphics[width=\linewidth, trim = {1cm 0.5cm 0.5cm 0.75cm}, clip]{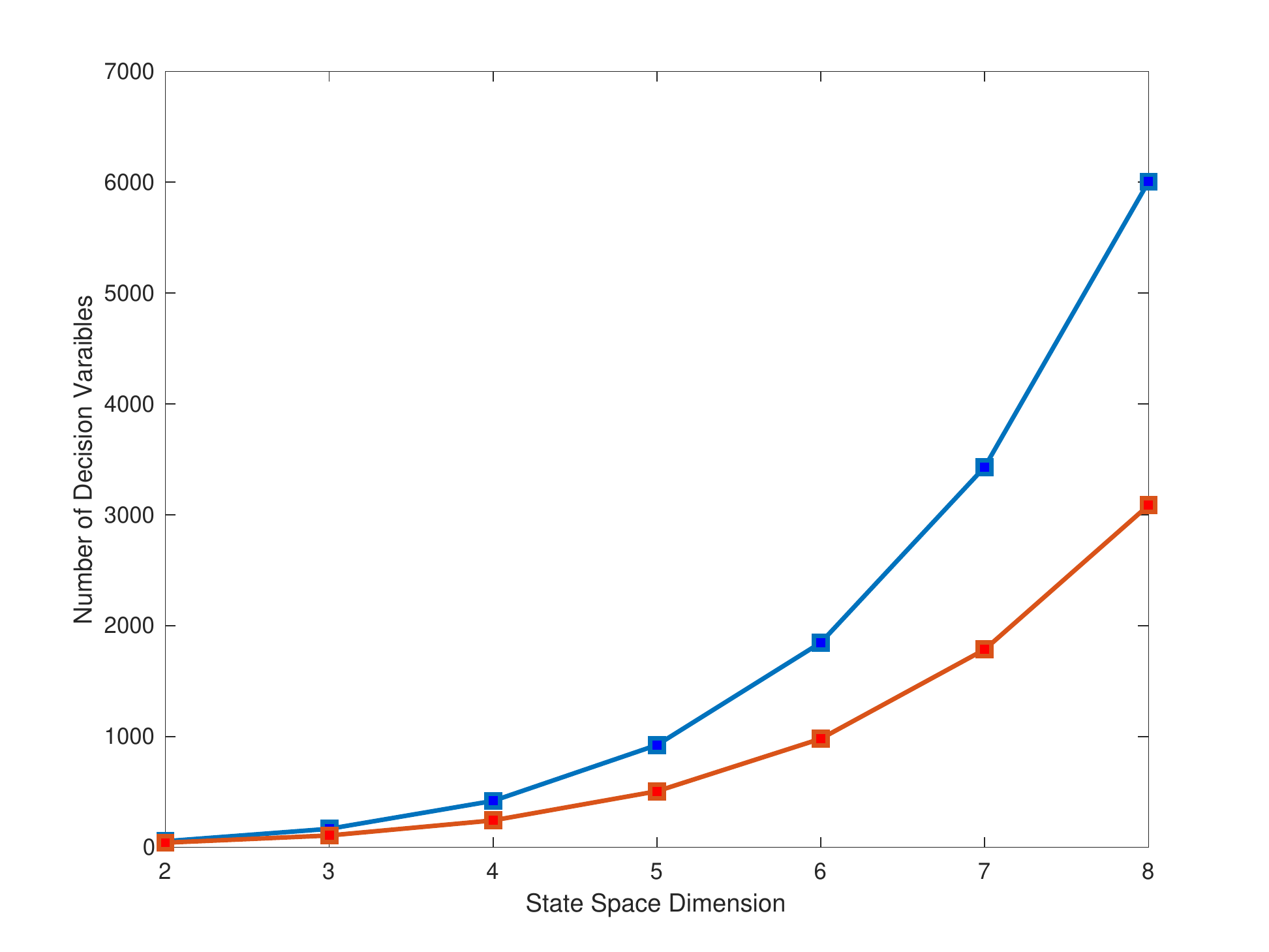}
	%	\vspace{-10pt}
		\caption{}
		%\caption{\footnotesize Graph showing that the number of decision variables in the underlying SDP problem of Opt.~\eqref{opt: Full local stab SOS}, plotted as the blue curve, is larger than that of Opt.~\eqref{opt: reduced sos} for ODE~\eqref{ODE: Lotka }, plotted as the red curve. }
	\label{fig: lokta}
	\end{subfigure}%
	\begin{subfigure}[t]{ 0.333 \textwidth}
			\includegraphics[width=\linewidth, trim = {1cm 0.5cm 0.5cm 0.65cm}, clip]{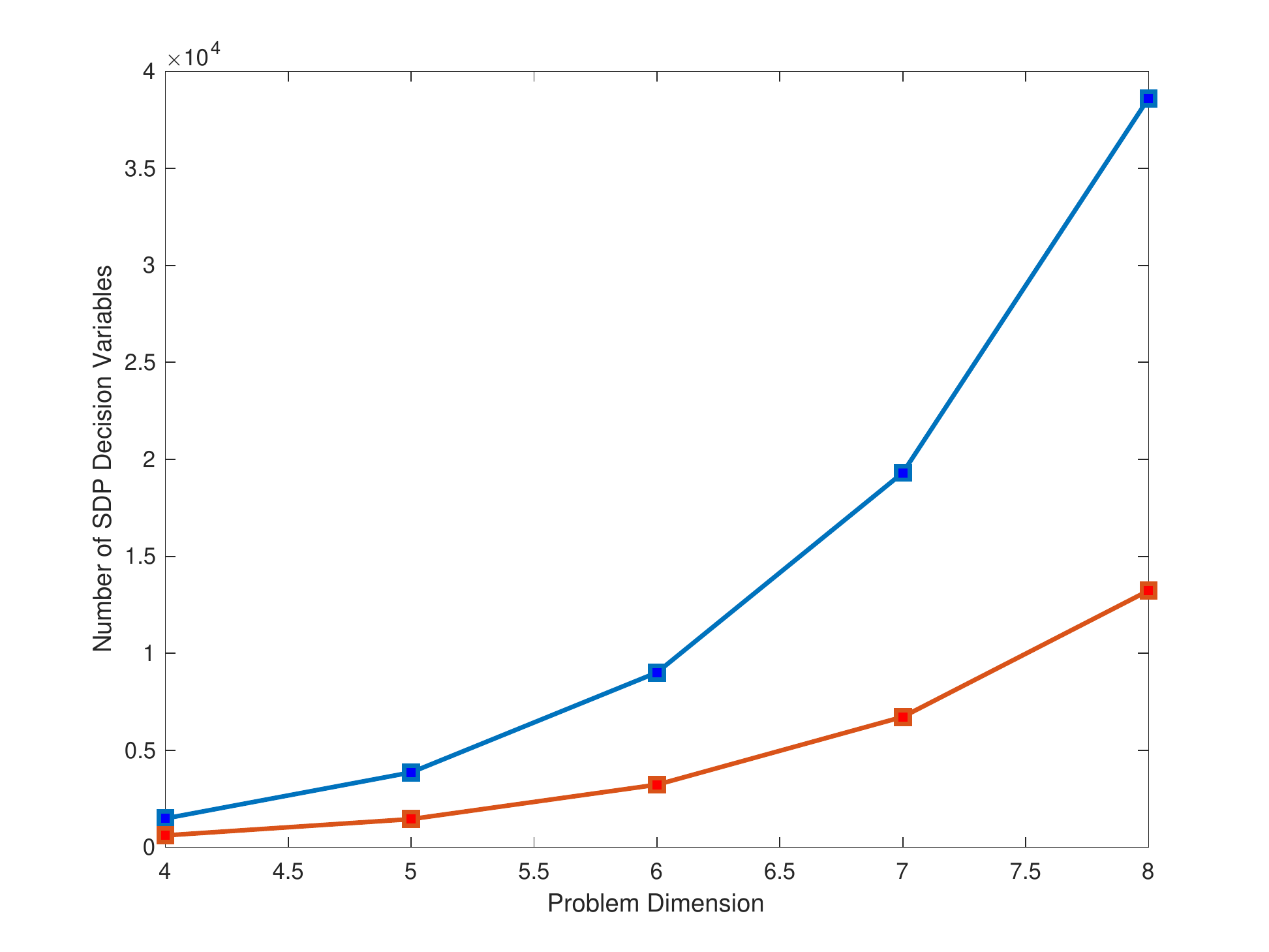}
%	\vspace{-10pt}
\caption{}
%	\caption{ \footnotesize Graph showing that the number of decision variables in the underlying SDP problem of Opt.~\eqref{opt: Full local stab SOS}, plotted as the blue curve, is larger than that of Opt.~\eqref{opt: reduced sos} for ODE~\eqref{ODE: Lotka }, plotted as the red curve. }
	\label{fig: coupled linear ODE}
\end{subfigure}%
	\vspace{-10pt}
	\caption{\footnotesize (\protect\subref{fig: Illistrative example}) Graph showing that the LF given in Eq.~\eqref{V: illistrative example} certifies the local stability of ODE~\eqref{ODE: illistrative example}. The center manifold, $y=x^2$, is also plotted as the dotted blue line. Several trajectories of the ODE for various initial conditions are plotted as the black curves.
(\protect\subref{fig: lokta}) 	Graph showing that for ODE~\eqref{ODE: Lotka } the number of decision variables in the underlying SDP problem of Opt.~\eqref{opt: Full local stab SOS}, plotted as the blue curve, is larger than that of Opt.~\eqref{opt: reduced sos}, plotted as the red curve.
(\protect\subref{fig: coupled linear ODE}) 	Graph showing that for ODE~\eqref{ODE: coupled linear} the number of decision variables in the underlying SDP problem of Opt.~\eqref{opt: Full local stab SOS}, plotted as the blue curve, is larger than that of Opt.~\eqref{opt: reduced sos}, plotted as the red curve.}
%	\label{figs}
	\vspace{-15pt}
\end{figure*}

%\begin{figure}
%	\includegraphics[scale=0.5, trim = {1.5cm 0 0 0cm}, clip]{lokta}
%	\vspace{-30pt}
%	\caption{Graph showing that the number of decision variables in the underlying SDP problem of Opt.~\eqref{opt: Full local stab SOS}, plotted as the blue curve, is larger than that of Opt.~\eqref{opt: reduced sos} for ODE~\eqref{ODE: Lotka }, plotted as the red curve. }
%	\vspace{-20pt} \label{fig: lokta}
%\end{figure}
%\begin{figure}
%	\includegraphics[scale=0.5, trim = {1.75cm 0 0 0cm}, clip]{ill_plot}
%	\vspace{-35pt}
%\caption{Graph showing that the LF given in Eq.~\eqref{V: illistrative example} certifies the local stability of ODE~\eqref{ODE: illistrative example} according to Lem.~\ref{lem: LF subset of ROA}. The center manifold, $y=x^2$, is also plotted as the dotted blue line. Several trajectories of the ODE for various initial conditions are plotted as the black curves. }
%	\vspace{-20pt} \label{fig: Illistrative example}
%\end{figure}

\section{Using SOS to Certify Local Stability} \label{sec: SOS problems}
\vspace{-0.1cm}
Consider the problem of certifying the local stability of an ODE~\eqref{eqn: ODE}, defined by some vector field $f$. This problem can be solved by using Lyapunov's second method (Thm.~\ref{thm: converse LF}). In cases where the vector field, $f$, is polynomial we can search for such a LF using SOS programming~\cite{Declan2022,James2015Lyapunov}. We can find such LFs by solving the following $2d$-degree SOS feasibility problem:
%\begin{align} \label{opt: Full local stab SOS}
%\text{Find: } & V \in \mathbb{R}_d[x] \text{ and }  s_1,s_2,s_3 \in \Sigma_{2d}  \text{ such that, }\\ \nonumber
% & V(0) =0  \\ \nonumber 
%& V(x) - s_0(x)(R- ||x||_2^2) = s_1(x) \text{ for all } x \in \R^n \\ \nonumber
% & -\nabla V(x)^\top f(x) - s_2(x)(R- ||x||_2^2) = s_3(x) \text{ for all } x \in \R^n,
%\end{align}
\vspace{-0.2cm} \begin{align} \nonumber
&\text{Find: }  V \in \mathbb{R}_{2d}[x],  s_1,s_2,s_3 \in \Sigma_{2d}   \text{ such that, }\\ \label{opt: Full local stab SOS}
 & V(0) =0,  \quad V(x) = s_1(x) + \eps  x^\top x \text{ for } x \in \R^n, \\ \nonumber
 & -\nabla V(x)^\top f(x) - s_2(x)(R^2- ||x||_2^2) = s_3(x) \text{ for } x \in \R^n,
\end{align}
where $R>0$ and $\eps>0$. Note that $R>0$ is included in Opt.~\eqref{opt: Full local stab SOS} so we only enforce $V$ to be a LF locally (over the ball $B_R(0)$). Also note that $\eps>0$ is included in Opt.~\eqref{opt: Full local stab SOS} to avoid the trivial solution $V(x) \equiv 0$. Typically $R>0$ and $\eps>0$ are selected to be small, for instance $R=\eps=0.1$.

If ODE~\eqref{eqn: ODE} is of the form given  in Eqs.~\eqref{ODE: zero real part} and~\eqref{ODE: negative real part} then Thm.~\ref{thm: seperable LF} indicates that we can certify local stability by searching for a partially quadratic LF of the form given in Eq.~\eqref{eq: Form of V}. This motivates the following $2d$-degree SOS feasibility problem:
\vspace{-0.25cm} \begin{align} \nonumber
&	\text{Find: }  J \in \mathbb{R}_{2d}[x_1], H_i \in \mathbb{R}_{2d}[x_1], P \in \R^{n \times n}, s_1,s_2,s_3 \in \Sigma_{2d}  \\ \label{opt: reduced sos}
 & \text{such that, } V(0,0) =0, \hspace{0.1cm} \\ \nonumber 
 & V(x_1,x_2) =s_1(x_1,x_2) + \eps (x_1,x_2)^\top (x_1,x_2) \text{ for } x \in \R^n, \\ \nonumber
	& -\nabla V(x_1,x_2)^\top  \hspace{-0.1cm}  \begin{bmatrix}
		A_1 x_1 + g_1(x_1,x_2)\\
		A_2 x_2 + g_2(x_1,x_2)
	\end{bmatrix}\\ \nonumber 
& \hspace{0.4cm} - s_2(x_1,x_2)(R^2- ||(x_1,x_2)||_2^2) = s_3(x_1,x_2) \text{ for } x \in \R^n,
\end{align}
where $ \hspace{-0.1cm} V  \hspace{-0.05cm}  (  \hspace{-0.05cm}  x_1,  \hspace{-0.05cm}  x_2  \hspace{-0.05cm}  ) \hspace{-0.1cm} =  \hspace{-0.1cm}  J \hspace{-0.05cm}( \hspace{-0.05cm} x_1 \hspace{-0.05cm})  \hspace{-0.1cm}  +  \hspace{-0.1cm}  x_2^T \hspace{-0.15cm} \begin{bmatrix}  H_1(x_1) \\ \vdots \\ H_{n-k}(x_1) \end{bmatrix}  \hspace{-0.2cm}  +  \hspace{-0.1cm}   x_2^\top \hspace{-0.1cm} P  \hspace{-0.05cm} x_2  $, $\hspace{-0.05cm} \eps \hspace{-0.1cm} > \hspace{-0.1cm} 0 \hspace{-0.05cm}$ and $ \hspace{-0.05cm}R \hspace{-0.1cm} >\hspace{-0.1cm}0$.

 {Note, Opt.~\eqref{opt: reduced sos} can certify the local asymptotic stability of ODEs of the form given in Eqs.~\eqref{ODE: zero real part} and~\eqref{ODE: negative real part}. General ODEs can be converted to be of this form using a coordinate change given in Eq.~\eqref{eq: block digaonal form}. This coordinate change can be numerically found using  Matlab functions \texttt{jordan} and \texttt{cdf2rdf}.}
 
 Searching for partially quadratic LFs by solving Opt.~\eqref{opt: reduced sos} as opposed to searching for fully non-quadratic LFs by solving Opt.~\eqref{opt: Full local stab SOS} results in computational savings due to the reduction in decision variables. These computational savings will be demonstrated through several numerical examples in the next section.
\vspace{-0.3cm}
\section{ {Numerical Examples}} \vspace{-0.1cm} \label{sec: numerical examples}
%We next numerically demonstrate the potential computational savings of our proposed method outlined in Section~\ref{sec: SOS problems}.
\begin{ex}[The Generalized Lotka–Volterra equations]
The competition of different groups (species, resources, etc) can be modelled by the following ODE,
\vspace{-0.2cm} \begin{align} \label{ODE: Lotka }
	&\dot{x}_i(t)=x_i(t) g(x(t)), \\ \nonumber
	& \text{ where } g(x)= r+ Bx, \quad  r \in \R^n  \text{ and } B \in R^{n \times n}.
\end{align}
%The size of group $i \in \{1,\dots,n\}$ at time $t \ge 0$ is denoted as $x_i(t) \in \R$. If $r_i>0$ then the size of group $i$ grows in the absence of other groups (for instance plants). If $r_i<0$ then the size of group $i$ decreases in the absence of other groups (for instance predators). If $r_i=0$ then the size group $i$ remains constant in the absence of others (for instance a closed eco system or some stored resource). The value of $B_{i,j}$ provides a metric on the impact group $j$ has on group $i$. The ODE~\eqref{ODE: Lotka } has two equilibrium points at $0 \in \R^n$ and $-B^{-1}r \in \R^n$. The equilibrium point at $0 \in \R^n$ corresponds to an extinction level event where all groups have size 0.

%Let us consider randomly generated values for $r$ and $B$ where $r_1=0$ and $r_i<0$ for all $i \in \{2,\dots,n\}$. Such a situation could occur where group $x_1$ represents a resource, like the population of trees in an ancient woodland, that remains fairly constant in the absence of the outside world. The other groups could represent factions of humans that would like to increase the woodland (when $B_{1,j}>0$) or remove the woodland (when $B_{1,j}<0$). In the Absence of other groups both fractions of  humans decrease ($r_i<0$) since there is no one left to fight with or any resource to fight over. We are interested in analysing the stability of the extinction level event ($x=0$).

 {Clearly,  ODE~\eqref{ODE: Lotka } has an equilibrium point at $0 \in \R^n$. The corresponding linearization matrix $A \in \R^{n \times n}$, given in Eq.~\eqref{eq: lineraization of f marix}, is such that $A_{i,j}=\begin{cases}
	r_i \text{ if } i=j\\
	0 \text{ otherwise. }
\end{cases}$ Let us consider randomly generated values for $r$ and $B$ where $r_1=0$ and $r_i<0$ for all $i \in \{2,\dots,n\}$. Hence $A$ has eigenvalues $\{r_1,\dots, r_n\}$ which in this case are either purely imaginative or have negative real part. Thus we are unable to  determine the stability of $x=0$ by Lyapunov's first method. Cor.~\ref{cor: only one imaginary eigenvalue implies partiall quad LF} shows this system is asymptotically stable iff there exists a partially quadratic LF. Solving Opts.~\eqref{opt: Full local stab SOS} and~\eqref{opt: reduced sos} for $n=2$ to $8$ at $d=6$, $R=0.01$ and $\eps=0.00001$, allows us to find feasible LFs in each case.  In Fig.~\ref{fig: lokta} we have plotted the difference in the number of decision variables associated with each optimization problem. For $n=8$ it took Yalmip~\cite{lofberg2004yalmip} and Mosek {213s} to solve Opt.~\eqref{opt: Full local stab SOS} and  157s to solve Opt.~\eqref{opt: reduced sos}.}
\end{ex}

\begin{ex}[{Stable linear systems with 	nonlinear interconnection}]
Let us consider the following ODE system,
\vspace{-0.2cm} \begin{align} \label{ODE: coupled linear}
	\dot{z}_1(t)\hspace{-0.1cm} = \hspace{-0.1cm} g(z_1(t), \hspace{-0.05cm}  z_2(t),  \hspace{-0.05cm} z_3(t)),  \\ \nonumber 
	\dot{z}_2(t) \hspace{-0.1cm}  = \hspace{-0.1cm}  Q_1 z_2(t),  \quad   \dot{z}_3(t) \hspace{-0.1cm}  = \hspace{-0.1cm}  Q_2 z_3(t),
\end{align}
where $g: \R \times \R^n \times \R^n \to \R$ is such that $\nabla g(0,0,0)=0$ and where all the eigenvalues of $Q_1 \in \R^{n_1 \times n_1}$ and $Q_2\in \R^{n_2 \times n_2}$ have negative real part. The linearization matrix, given in Eq.~\eqref{eq: lineraization of f marix}, for this system is then $A:=\begin{bmatrix} 0&0  & 0\\ 0 & Q_1 &   0 \\ 0 &0  & Q_2 \end{bmatrix}.$
This matrix is already in the block diagonal form of Eq.~\eqref{eq: block digaonal form} with $A_1=0$ and $A_2=\begin{bmatrix}
	Q_1 & 0 \\ 0 & Q_2
\end{bmatrix}$. Since $A_1$ is one dimensional by Cor.~\ref{cor: only one imaginary eigenvalue implies partiall quad LF} the system is locally asymptotically stable iff there exists a partially quadratic LF for this system. 

For simplicity we will consider the case $g(z_1,z_2,z_3)=z_1^2 + z_2^\top z_2 + z_3^\top z_3$, $Q_1=-\begin{bmatrix}
	1& 0 \\ 0 & 1
\end{bmatrix}$ and $Q_2= -I \in \R^{n \times n}$. For this ODE, $d=8$, $R=0.05$ and $\eps=0.0001$ we solve Opts.~\eqref{opt: Full local stab SOS} and~\eqref{opt: reduced sos} for $n=1$ to $5$, finding a LF in each case. Fig.~\ref{fig: coupled linear ODE} shows how the number of decision variables grows for each problem. For $n=5$ it took Yalmip~\cite{lofberg2004yalmip} and Mosek {12645s} to solve Opt.~\eqref{opt: Full local stab SOS} and  10398s to solve Opt.~\eqref{opt: reduced sos}.
\end{ex}

\vspace{-0.25cm} \section{Conclusion} \label{sec:conclusion}
\vspace{-0.2cm}
We have proposed conditions for which there exists a partially quadratic LF that can certify the local asymptotic stability of nonlinear ODEs. The existence proof was non-constructive, relying on the existence of the center manifold. However, knowledge of the existence of partially quadratic LFs  allows us to tighten our search of LFs, providing computational savings. This paper opens up many directions of future work such as investigating the conditions for which there exists a SOS partially quadratic LF and the conditions under which the proposed methods can be extended to global stability analysis.
\vspace{-0.9cm}
\bibliographystyle{unsrt}
\bibliography{bib_cm}

\begin{thebibliography}{10}

\bibitem{lorenz1963deterministic}
Edward~N Lorenz.
\newblock Deterministic nonperiodic flow.
\newblock {\em Journal of atmospheric sciences}, 20(2):130--141, 1963.

\bibitem{jansen2000local}
Vincent~AA Jansen and Alun~L Lloyd.
\newblock Local stability analysis of spatially homogeneous solutions of
  multi-patch systems.
\newblock {\em Journal of mathematical biology}, 41(3):232--252, 2000.

\bibitem{https://doi.org/10.48550/arxiv.2111.09382}
Lucas Lugnani, Morgan Jones, Luís F.~C. Alberto, Mathew Peet, and Daniel
  Dotta.
\newblock Combining trajectory data with energy functions for improved region
  of attraction estimation, 2021.

\bibitem{Declan2022}
Declan Jagt, Sachin Shivakumar, Peter Seiler, and Matthew Peet.
\newblock Efficient data structures for exploiting sparsity and structure in
  representation of polynomial optimization problems: Implementation in
  sostools, 2022.

\bibitem{ahmadi2019dsos}
Amir~Ali Ahmadi and Anirudha Majumdar.
\newblock Dsos and sdsos optimization: more tractable alternatives to sum of
  squares and semidefinite optimization.
\newblock {\em SIAM Journal on Applied Algebra and Geometry}, 3(2):193--230,
  2019.

\bibitem{kundu2015sum}
Soumya Kundu and Marian Anghel.
\newblock A sum-of-squares approach to the stability and control of
  interconnected systems using vector lyapunov functions.
\newblock In {\em 2015 American Control Conference (ACC)}, pages 5022--5028.
  IEEE, 2015.

\bibitem{anderson2011decomposition}
James Anderson and Antonis Papachristodoulou.
\newblock A decomposition technique for nonlinear dynamical system analysis.
\newblock {\em IEEE Transactions on Automatic Control}, 57(6):1516--1521, 2011.

\bibitem{schlosser2020sparse}
Corbinian Schlosser and Milan Korda.
\newblock Sparse moment-sum-of-squares relaxations for nonlinear dynamical
  systems with guaranteed convergence.
\newblock {\em arXiv preprint arXiv:2012.05572}, 2020.

\bibitem{ito2012capability}
Hiroshi Ito, Sergey Dashkovskiy, and Fabian Wirth.
\newblock Capability and limitation of max-and sum-type construction of
  lyapunov functions for networks of iiss systems.
\newblock {\em Automatica}, 48(6):1197--1204, 2012.

\bibitem{ali2020computational}
M~Ali Al-Radhawi, David Angeli, and Eduardo~D Sontag.
\newblock A computational framework for a lyapunov-enabled analysis of
  biochemical reaction networks.
\newblock {\em PLoS computational biology}, 16(2):e1007681, 2020.

\bibitem{tan2008stability}
Weehong Tan and Andrew Packard.
\newblock Stability region analysis using polynomial and composite polynomial
  lyapunov functions and sum-of-squares programming.
\newblock {\em IEEE Transactions on Automatic Control}, 53(2):565--571, 2008.

\bibitem{tacchi2020approximating}
Matteo Tacchi, Carmen Cardozo, Didier Henrion, and Jean~Bernard Lasserre.
\newblock Approximating regions of attraction of a sparse polynomial
  differential system.
\newblock {\em IFAC-PapersOnLine}, 53(2):3266--3271, 2020.

\bibitem{cibulka2021spatio}
Vit Cibulka, Milan Korda, and Tom{\'a}{\v{s}} Hani{\v{s}}.
\newblock Spatio-temporal decomposition of sum-of-squares programs for the
  region of attraction and reachability.
\newblock {\em IEEE Control Systems Letters}, 6:812--817, 2021.

\bibitem{rantzer2013separable}
Anders Rantzer, Bj{\"o}rn~S R{\"u}ffer, and Gunther Dirr.
\newblock Separable lyapunov functions for monotone systems.
\newblock In {\em 52nd IEEE Conference on Decision and Control}, pages
  4590--4594. IEEE, 2013.

\bibitem{Khalil_1996}
H~Khalil.
\newblock {\em Nonlinear Systems}.
\newblock 1996.

\bibitem{haasdonk2021kernel}
Bernard Haasdonk, Boumediene Hamzi, Gabriele Santin, and Dominik Wittwar.
\newblock Kernel methods for center manifold approximation and a weak
  data-based version of the center manifold theorem.
\newblock {\em Physica D: Nonlinear Phenomena}, 427:133007, 2021.

\bibitem{bacciotti2005liapunov}
Andrea Bacciotti and Lionel Rosier.
\newblock {\em Liapunov functions and stability in control theory}.
\newblock Springer Science \& Business Media, 2005.

\bibitem{williams2007linear}
Robert~L Williams, Douglas~A Lawrence, et~al.
\newblock {\em Linear state-space control systems}.
\newblock John Wiley \& Sons, 2007.

\bibitem{atassi1999separation}
Ahmad~N Atassi and Hassan~K Khalil.
\newblock A separation principle for the stabilization of a class of nonlinear
  systems.
\newblock {\em IEEE Transactions on Automatic Control}, 44(9):1672--1687, 1999.

\bibitem{vidyasagar2002nonlinear}
M.~Vidyasagar.
\newblock {\em Nonlinear Systems Analysis: Second Edition}.
\newblock Classics in Applied Mathematics. Society for Industrial and Applied
  Mathematics (SIAM, 3600 Market Street, Floor 6, Philadelphia, PA 19104),
  2002.

\bibitem{roberts1997low}
AJ~Roberts.
\newblock Low-dimensional modelling of dynamical systems.
\newblock {\em arXiv preprint chao-dyn/9705010}, 1997.

\bibitem{James2015Lyapunov}
James Anderson and Antonis Papachristodoulou.
\newblock Advances in computational lyapunov analysis using sum of squares
  programming.
\newblock {\em Discrete and Continuous Dynamical Systems B}, 20(8):2361--2381,
  2015.

\bibitem{lofberg2004yalmip}
Johan Lofberg.
\newblock Yalmip: A toolbox for modeling and optimization in matlab.
\newblock In {\em 2004 IEEE international conference on robotics and automation
  (IEEE Cat. No. 04CH37508)}, pages 284--289. IEEE, 2004.

\bibitem{lang1987linear}
Serge Lang.
\newblock {\em Linear algebra}.
\newblock Springer Science \& Business Media, 1987.

\end{thebibliography}
\vspace{-0.35cm}
\appendix
\vspace{-0.25cm}
\begin{lem} \label{lem: block diagonal matrix}
	Suppose the matrix $A \in \R^{ n \times n}$ has distinct eigenvalues $\{ \lambda_1,...,\lambda_p \} \subset \mathbb{C}$ for some $1<p \le n$. If sets $S_1, S_2 \subset \mathbb{C}$ are such that $S_1 \cup S_2 = \{ \lambda_1,...,\lambda_p \}$, $S_1 \cap S_2 = \emptyset$, and if $\lambda \in S_i$ then $\bar{\lambda} \in S_i$ for $i=1,2$. Then there exists a non-singular matrix $T \in \R^{n \times n}$ such that
	\vspace{-0.2cm} \begin{align*}
		TAT^{-1}= \begin{bmatrix} A_1 &   0 \\ 0   & A_2 \end{bmatrix} \in \R^{n \times n},
	\end{align*}
	where the set of eigenvalues of $A_1 \in \R^{k \times k}$ is equal to $S_1$, the set of eigenvalues of $A_2 \in \R^{(n-k) \times (n-k)}$ is equal to $S_2$, $k=\sum_{\lambda \in S_1}Dim \bigg(Ker\bigg( (\lambda I - A)^{m(\lambda)} \bigg) \bigg)$ and $m(\lambda)$ is the algebraic multiplicity of eigenvalue $\lambda$.
\end{lem}
\begin{proof} Apply Theorem~4.2 Page~257 from~\cite{lang1987linear}.
	\end{proof}

%\begin{lem} \label{lem: calculus identities}
%	Consider the functions $f_1(x) = x^\top b$ and $f_2(x)=x^\top P x$, where $b \in \R^n$ and $P \in \R^{n \times n}$. The following calculus identities hold,
%	\begin{align*}
%		\nabla f_1(x) & =b. \\
%		\nabla f_2(x) & = x^\top (P + P^\top).
%	\end{align*}
%Moreover, if $P$ is a symmetric matrix then $\nabla f_2(x)=2x^\top P$.
%\end{lem}

\begin{lem} \label{lem: Lip bound}
Consider $V \in C^1(\R^n,\R)$. Let $K:=\sup_{x \in B_R(0)} ||\nabla V(x)||_2< \infty$. Then
\vspace{-0.15cm} \begin{align} \label{eq: lip cts}
	||V(x)-V(y)||_2 \le K ||x-y||_2 \text{ for all } x,y \in B_R(0).
\end{align}
Furthermore, if $\nabla V(0)=0$ then for any $\delta>0$ there exists $R>0$ such that 
	\vspace{-0.2cm}\begin{align}  \label{eq: lip cts arbitrarily small}
	||V(x)-V(y)||_2 \le \delta ||x-y||_2 \text{ for all } x,y \in B_R(0).
\end{align}
\end{lem}
\begin{proof}
	By the Mean Value Theorem for any $x,y \in B_R(0)$ there exists $c \in (0,1)$ such that 
		\vspace{-0.2cm} \begin{align*}
		||V(x)-V(y)||_2  \le & ||\nabla V(cx + (1-c)y)||_2  ||x-y||_2.
	\end{align*}
Then letting $K:=\sup_{x \in B_R(0)} ||\nabla V(x)||_2$ it follows that $||\nabla V(cx + (1-c)y)||_2 \le K$ for all $x,y \in B_R(0)$ and $c \in (0,1)$. Hence, Eq.~\eqref{eq: lip cts} follows.

Now suppose $\nabla V(0)=0$. Since $V \in C^1(\R^n,\R)$ it follows that $F \in C(\R^n,\R)$, where $F(x):=||\nabla V(x)||_2$. Then for any $\delta>0$ there exists $R>0$ such that $|F(0)-F(x)|<\delta/2$ for all $||x||_2<R^2$. Thus it follows,
	\vspace{-0.2cm}	\begin{align*}
||\nabla V(x)||_2 < \delta /2 \text{ for all } x \in B_R(0),
\end{align*}
implying $K:=\sup_{x \in B_R(0)} ||\nabla V(x)||_2< \delta$. Hence, Eq.~\eqref{eq: lip cts arbitrarily small} holds.
	\end{proof}

\begin{thm}[The Center Manifold Theorem~\cite{Khalil_1996}] \label{thm: The Center Manifold Theorem} 
	Consider an ODE given by the Eqs.~\eqref{ODE: zero real part} and~\eqref{ODE: negative real part} where $\frac{\partial}{\partial x_i} g_j(0)=0$ for $i \in \{1,\dots,n\}$ and $j \in \{1,2\}$. There exists $R>0$ and a function $\eta: C^{\infty}(B_R(0), \R^{n-k})$ such that 
	\begin{itemize}
		%		\item If $z \in \{(x_1,x_2) \in \R^k \times \R^{n-k}: x_2 = \eta(x_1) \}$ then $[\phi_{\tilde{f}}(z,t)]_2 = \eta([\phi_{\tilde{f}}(z,t)]_1  )$ for all $t \ge 0$.
		\item The function $\eta$ is such that $\eta(0)=0$ and $\frac{\partial}{\partial y_i} \eta (0)=0$.
		\item The function $\eta$ solves the following system of PDEs,
	\end{itemize}
	\vspace{-0.2cm} \begin{align} \nonumber
	& A_2 \eta (y) + g_2(x,\eta(y)) - \nabla \eta(y)^\top (A_1 y + g_1(y, \eta(y)))=0 \\ \label{PDE: eta}
	& \hspace{1.5cm} \text{ for all } y \in \{x \in \R^{n-k} : ||x||_2<R\}.
\end{align}
\end{thm}

%Close to the origin the ODE evolves on its center manifold and reduces to a system of lower state space dimension. As stated next, the ODE is locally asymptotically stable if and only if the ``reduced" system is locally asymptotically stable.
\begin{prop}[\cite{Khalil_1996}]  \label{prop: Stability of reduced system} 
	The ODE given by the Eqs.~\eqref{ODE: zero real part} and~\eqref{ODE: negative real part} where $\frac{\partial}{\partial x_i} g_j(0)=0$ for $i \in \{1,\dots,n\}$ and $j \in \{1,2\}$ is locally asymptotically stable if and only if the following ODE is locally asymptotically stable,
	\begin{align} \label{ODE: reduced}
		\dot{z}_1(t)= A_1 z_1(t) + g_1(z_1(t), \eta(z_1(t))),
	\end{align}
	where $\eta$ solves the PDE~\eqref{PDE: eta} (known to exist by Theorem~\ref{thm: The Center Manifold Theorem}).
\end{prop}

\end{document}